\documentclass[a4paper,10pt]{article}
\usepackage{latexsym,amssymb,amsfonts,amsmath,amsthm,amstext,color,graphicx,times}
\usepackage[mathscr]{euscript}
\usepackage{indentfirst}
\usepackage{cite}
\usepackage{mathrsfs}
\usepackage{subfigure}

\usepackage{amsxtra}
\usepackage{dsfont}
\usepackage{stmaryrd}

\newcommand{\R}{\mathbb{R}}

\usepackage{wrapfig}
\usepackage[all]{xy}
\usepackage{tikz}
\usepackage{hyperref}
\usepackage{animate}
\usetikzlibrary{patterns}

\theoremstyle{plain}
\newtheorem{theorem}{Theorem}[section]

\newtheorem{proposition}[theorem]{Proposition}
\newtheorem{corollary}[theorem]{Corollary}

\theoremstyle{definition}

\newtheorem{example}[theorem]{Example}

\theoremstyle{remark}
\newtheorem{remark}[theorem]{Remark}

\newcommand{\tos}{\rightrightarrows} 
\DeclareMathOperator*{\argmin}{arg\,min}

\DeclareMathOperator*{\co}{co}

\DeclareMathOperator*{\nep}{NEP}
\DeclareMathOperator*{\gnep}{GNEP}

\DeclareMathOperator*{\psgnep}{PSGNEP}

\DeclareMathOperator{\gra}{gra}

\title{Remarks on projected solutions for  generalized Nash games}

\author{ 
Carlos Calder\'on 
\thanks{Instituto de Matem\'atica y Ciencias Afines. Lima, Per\'u. Email: carlos.calderon@imca.edu.pe} 
\and
John Cotrina  
\thanks{Universidad del Pac\'ifico, Lima, Per\'u. Email: cotrina\_je@up.edu.pe} 
}

\begin{document}

\maketitle

\begin{abstract}
In this work, we focus on the concept of projected solutions for generalized Nash equilibrium problems. We present new existence results by considering sets of strategies that are not necessarily compact. The relationship between projected solutions and Nash equilibria is studied for the generalized Nash game proposed by Rosen. Finally, we demonstrate that every projected solution of a game is associated with a Nash equilibrium, but in a different game.
\bigskip


\bigskip


\end{abstract}

\section{Introduction}\label{intro}

Nash games \cite{Nash} focus on the strategic interaction between two or more players, where each player chooses a strategy and seeks to maximize their own outcome, taking into account the choices of the other players. A Nash equilibrium is reached when no player can achieve a better outcome by changing their strategy, provided that the other players maintain their strategies.  Thus, the importance of finding suitable assumptions in order to guarantee the existence of Nash equilibria was taken into account for many researchers, see for instance \cite{Dasgupta,MORGAN2007,Rabia2016,Reny,Parilina2022-oi} and their references.

On the other hand, the generalized Nash equilibrium problem was first proposed by Arrow and Debreu \cite{Arrow-Debreu}, who referred to it as an "abstract economy." These generalized games extend the focus of Nash games by considering more general scenarios where players can have different sets of available strategies and different preferences over outcomes. In these games, the objective is to find a generalized equilibrium that takes into account the constraints and individual preferences of the players. The concept of generalized equilibrium has been fundamental in economic theory in analyzing the efficient allocation of resources and the maximization of social welfare \cite{Debreu, K-2005, K-2016, CK-2015, Vardar2019}. A particular generalized Nash game was introduced by Rosen in \cite{Rosen}, but it was in the 1990s that many authors began addressing these generalized games in order to establish sufficient conditions to guarantee the existence of generalized Nash equilibria \cite{HARKER199181, Cavazzuti, FACCHINEI2007159, Aussel-Dutta, AVS21, ASV-2016, AGM-16, CHS-2020, CZ-2018}.

Currently, generalized Nash games are being used by researchers to model electricity markets, as seen in \cite{ASV-2016, Le_Cadre2019-xc}. However, in 2016, Aussel, Sultana, and Vetrivel \cite{ASV-2016} showed that these electricity market problems may not have generalized Nash equilibria because the strategies of each player depend on the strategies of their rivals, but these strategies do not necessarily fall within a fixed set of strategies. This absence of generalized equilibria has led to the introduction of a new concept called projected solution for generalized Nash equilibrium problems. In addition, the authors in \cite{ASV-2016} reformulated the generalized Nash game as a quasi-variational inequality problem to obtain projected solutions, assuming convexity and differentiability. Later, Cotrina and Z\'u\~niga \cite{CZ-2019} extended the result given in \cite{ASV-2016} by considering continuity instead of differentiability, by reformulating these generalized games as quasi-equilibrium problems. Similarly, Castellani et al. \cite{CGL-2023} extended the main result in \cite{CZ-2019} by relaxing the compactness assumption of each strategy set. However, all the above results require the convexity assumption. Recently, Bueno and Cotrina \cite{BC-2021} established an existence result on projected solutions in the setting of quasi-convexity, considering a weak notion of continuity. Moreover, in \cite{BC-2021} the authors showed that both the quasi-variational inequality problem and the quasi-equilibrium problem can be reformulated as a certain generalized Nash equilibrium problem. The aim of this manuscript is to prove the existence of projected solutions for generalized Nash games in the setting of quasi-convexity, which is independent of the one given in \cite{BC-2021}. Moreover, we aim to study the relationship between projected solutions and generalized Nash equilibria, first for the generalized Nash game proposed by Rosen and later for the general case.

The paper is organized as follows. In Section \ref{preli}, we provide some definitions and notations. Section \ref{Main-section} is divided into three subsections. In the first subsection, we establish two existence results on projected solutions. In the second subsection, we focus on the generalized Nash game proposed by Rosen. Finally, in the third subsection, we reformulate the problem of finding projected solutions as a certain generalized Nash game.

\section{Preliminaries}\label{preli}
From now on $\|\cdot\|$ denotes a norm in $\R^n$. Given a subset $A$ of $\R^n$, we denote by $\co(A)$ the convex hull of $A$ and by $\overline{A}$ the closure of $A$.
For each $z\in \R^n$, we denote by $P_A(z)$ the projection of $z$ onto $A$, that is
\[
P_A(z)=\{w\in A:~\|z-w\|\leq \|z-x\|\mbox{ for all }x\in A\}.
\]

We now recall continuity notions for set-valued maps. 
Let $T: X\tos Y$ be a set-valued map with $X$ and $Y$ two topological spaces.
The map $T$ is said  to be \emph{closed}, when $\gra(T):=\big\{(x,y)\in X\times Y\::\: y\in T(x)\big\}.$ is a closed subset of $X\times Y$. Moreover, the map $T$ is \emph{lower semicontinuous} at $x\in X$ if for each open set $V$ such that $T(x_0)\cap V\neq\emptyset$ there exists $\mathscr{V}_x$ neighbourhood of $x$ such that $T(x')\cap V\neq\emptyset$ for every $x'\in \mathscr{V}_x$; it is \emph{upper semicontinuous} at $x\in X$ if  for each open set $V$, with $T(x)\subset V$,  there exists $\mathscr{V}_x$ neighbourhood of $x$ such that $T(x') \subset V$ for every $x'\in \mathscr{V}_x$. Finally, we say that the map $T$ is  \emph{continuous} when it is upper and lower semicontinuous.

It is known that the projection onto $A\subset\R^n$, $P_A$, defines a set-valued map from $\R^n$ onto $A$. 
 
\, 
 
We recall the notion of pseudo-continuity \cite{MORGAN2007} for functions.
A real-valued function $f:X\to\R$, where $X$ is a topological space, is said to be \emph{upper pseudo-continuous} if, for any $x,y\in X$  such that $f(x)<f(y)$, there exists a neighbourhood $\mathscr{V}_x$ of $x$ satisfying
\[
f(x')<f(y),\mbox{ for all }x'\in V_x.
\]
Moreover, the function $f$ is \emph{lower pseudo-continuous} if, $-f$ is upper pseudo-continuous. Finally, $f$ is said to be  \emph{pseudo-continuous} if, it is lower and upper pseudo-continuous.

It is important to notice that any upper semi-continuous function is upper pseudo-continuous, but the converse is not true in general, see \cite{Co21} and its references for more details on pseudo-continuity.  

\section{The generalized Nash equilibrium problem}\label{Main-section}

The Nash equilibrium problem (NEP in short) \cite{Nash} consists of a finite number of players, where each player has a strategy set and an objective function depending not only on his/her decision but also on the decision of his/her rivals. Formally, 
let $N$ be the set of players which is any finite and non-empty set.
Let us assume that each player 
$\nu \in N$ chooses a strategy $x^\nu $ in a strategy set $K_{\nu}$, which is a subset of  $\R^{n_\nu}$. We denote by $\R^n$, $K$ and $K_{-\nu}$ the Cartesian products of $\prod_{\nu\in N}\R^{n_\nu}$,  $\prod_{\nu\in N} K_{\nu}$ and $\prod_{\mu\in N\setminus\{\nu\}} K_{\mu}$, respectively. We can write $x = (x^\nu, x^{-\nu}) \in K$ in order to emphasize the strategy of player $\nu$, $x^\nu\in K_\nu$, and the strategy of the other players $x^{-\nu}\in K_{-\nu}$.

Given the strategy the players except of player $\nu$, $x^{-\nu}$, the player $\nu$ chooses a strategy $x^\nu$ such that it solves the following optimization problem  
\begin{align}\label{NEP}
\min \theta_\nu(z^\nu,x^{-\nu}),~\mbox{ subject to }~z^\nu\in K_\nu,
\end{align}
where $\theta_\nu: \R^n\to\R$ is a real-valued function and  $\theta_\nu(x^\nu,x^{-\nu})$ denotes the loss of the player $\nu$ suffers when the rival players have chosen 
the strategy $x^{-\nu}$ and he/she takes $x^\nu$.  Thus, a  {\em Nash equilibrium}  is a vector $\hat{x}\in K$ such that $\hat{x}^\nu$ solves \eqref{NEP} when the rival players take the strategy $\hat{x}^{-\nu}$, for any $\nu$.  We denote by $\nep(\{\theta_\nu,K_\nu\}_{\nu\in N})$ the set of Nash equilibria. 

\,

Arrow and Debreu \cite{Arrow-Debreu} dealt with a more complex situation where the strategy set of each player also depends on the decision of his/her rivals. Nowadays, these kind of games are called the generalized Nash equilibrium problem (GNEP in short). Thus, in a GNEP  each player $\nu$ has a strategy must belong to a set $X_\nu(x)\subset K_\nu$ that depends of all strategies. The aim of player $\nu$, given the others players' strategies $x^{-\nu}$, is to choose a strategy $x^\nu$ such that  it solves the next minimization problem
\begin{align}\label{GNEP}
\min\theta_\nu(z^\nu,x^{-\nu}),~\mbox{ subject to }~z^\nu\in X_\nu(x).
\end{align}
Thus, a vector $\hat{x}\in K$ is a \emph{generalized Nash equilibrium} if, $\hat{x}^\nu$ solves \eqref{GNEP} when the rival players take the strategy $\hat{x}^{-\nu}$, for any $\nu$. We denote by $\gnep(\{\theta_\nu,X_\nu\}_{\nu\in N})$ the set of
generalized Nash equilibria. 

It is clear that $\hat{x}\in \gnep(\{\theta_\nu,X_\nu\}_{\nu\in N})$ if, and only if, $\hat{x}\in\nep(\{\theta_\nu,X_\nu(\hat{x})\}_{\nu\in N})$.
Furthermore, observe that for a GNEP, the constraint maps $X_\nu:K\tos K_\nu$ induce a set-valued map $\mathcal{X}:K\tos K$ defined as
\[
\mathcal{X}(x)=\prod_{\nu\in N}X_\nu(x),
\]
which is a self-map, i.e. $\mathcal{X}(K)\subset K$. Consequently, any generalized Nash equilibrium is a fixed point of $\mathcal{X}$. 

\,

Aussel \emph{et al.}  \cite{ASV-2016} considered the more general situation, they assume that each constraint map $X_\nu$ is defined from $K$ onto $\R^{n_\nu}$ instead $K_\nu$. In this case, a vector $\hat{x}\in K$ is said to be a projected solution if, there exists $\hat{y}\in \mathcal{X}(\hat{x})$ such that the following two conditions hold:
\begin{enumerate}
\item $\hat{x}$ is a projection of $\hat{y}$ onto $K$;
\item $\hat{y}\in \nep(\{\theta_\nu, X_\nu(\hat{x})\}_{\nu\in N})$.
\end{enumerate}
Clearly, any generalized Nash equilibrium is a projected solution, but the converse is not true, see Remark 3.1, part 2 in \cite{BC-2021}.

\,

We divide this section in three parts, the first one is related to the existence of projected solutions, the second one is concerning to the generalized Nash game proposed by Rosen, and finally the third subsection consists to reformulate the problem of finding projected solutions to a particular GNEP.
\subsection{Existence result}

Before to establish our first result we need the following proposition, which is a consequence of the maximum theorem.
\begin{proposition}\label{ApplyBerge}
Let $X,Y,Z$ be three topological spaces, $T:X\tos Y$ be a set-valued map and $f:Y\times Z\to\R$ be a function. If $f$ is pseudo-continuous and $T$ is continuous with compact and non-empty values; then the map $M:X\times Z\tos Y$ defined as
\[
M(x,z)=\{y\in T(x):~f(y,z)\leq f(w,z)\mbox{ for all }w\in T(x)\}
\]
is upper semicontinuous with compact and non-empty values.
\end{proposition}
\begin{proof}
By considering $\hat{T}:X\times Z\tos Y$ and $\hat{f}:(X\times Z)\times Y\to\R$ defined as
\[
\hat{T}(x,z)=T(x)\mbox{ and }\hat{f}(x,z,y)=f(y,z).
\]
Clearly $\hat{f}$ is a pseudo-continuous, and $\hat{T}$ is continuous with compact and non-empty values. Thus, by Theorem 3.4 in \cite{Co21}, the map $M$ is upper semicontinuous with compact and non-empty values. 
\end{proof}

Now, we are in position to state our first existence result, which generalizes Theorem 4.2 in \cite{ASV-2016}. 
\begin{theorem}\label{MS}
Assume any norm in $\R^n$, and moreover for each player $\nu\in N$:
\begin{enumerate}
\item $K_\nu$ is convex, compact and non-empty subset of $\R^{n_\nu}$,
\item $X_\nu$ is continuous with convex, compact and non-empty values, 
\item $\theta_\nu$ is pseudo-continuous and
\item $\theta_\nu(\cdot,x^{-\nu})$ is quasi-convex, for all $x^{-\nu}$;
\end{enumerate}
then there exists a projected solution.
\end{theorem}
\begin{proof}
The projection map $P_K$ is upper semicontinuous with compact, convex and non-empty values, see \cite{Deutsch1973}. 

For each $\nu\in N$, consider the sets $D_\nu=\co(X_\nu(K))$. In addition, we also consider  the sets $D=\prod_{\nu\in N} D_\nu$ and 
$C=\co(P_K(D))$, which are convex, compact and non-empty.
For each $\nu\in N$, we define the map $M_\nu:C\times D\tos D_\nu$ as
\[
M_\nu(x,y)=\{z^\nu\in X_\nu(x):~\theta_\nu(z^\nu,y^{-\nu})\leq \theta_\nu(w^\nu,y^{-\nu})\mbox{ for all }w^\nu\in X_\nu(x)\},
\] 
which is upper semicontinuous with compact and non-empty values, due to Proposition \ref{ApplyBerge}. Moreover, $M_\nu$ is convex-valued because $\theta_\nu$ is quasiconvex concerning to its player's variable. Thus, the map $M:C\times D\tos D$ defined as
\[
M(x,y)=\prod_{\nu\in N}M_\nu(x,y)
\]
is upper semicontinuous with convex, compact and non-empty values. On the other hand,
by considering the map $R:D\tos C\times D$ defined as
\[
R(y)=P_K(y)\times \{y\},
\]
which is clearly upper semicontinuous with convex, compact and non-empty values. Consequently the map $M\circ R:D\tos D$ is Kakutani's factorizable, and by Lassonde's fixed point theorem \cite{Lassonde}, that means there exists $\hat{y}\in D$ such that $\hat{y}\in M\circ R(\hat{y})$. Thus, there exists $\hat{x}\in C$ such that $\hat{x}\in P_K(\hat{y})$ and
$\hat{y}\in M(\hat{x},\hat{y})$. Now, $\hat{y}\in M(\hat{x},\hat{y})$ if, and only if, for each $\nu$, we have
\[
\theta_\nu(\hat{y})\leq \theta_\nu(w^\nu,\hat{y}^{-\nu})\mbox{ for all }w^\nu\in X_\nu(\hat{x}).
\]
Therefore, $\hat{x}$ is a projected solution.
\end{proof}

As a direct consequence of the previous result we have the following corollary, which is slight modification of Arrow and Debreu result \cite{Arrow-Debreu}.
\begin{corollary}\label{cor-PS}
Assume that for each player $\nu\in N$:
\begin{enumerate}
\item $K_\nu$ is convex, compact and non-empty subset of $\R^{n_\nu}$,
\item $X_\nu:K\tos K_\nu$ is continuous with convex, compact and non-empty values, 
\item $\theta_\nu$ is pseudo-continuous and
\item $\theta_\nu(\cdot,x^{-\nu})$ is quasi-convex, for all $x^{-\nu}$;
\end{enumerate}
then the set $\gnep(\{\theta_\nu,X_\nu\}_{\nu\in N})$ is non-empty.

\end{corollary}

The following example says that Theorem \ref{MS} is not a direct consequence of the one given by Bueno and Cotrina in \cite{BC-2021}.
\begin{example}
Consider $K_1=K_2=[0,1]$ and the maps $X_1,X_2:K\tos  \R$ defined as
\[
K_1(x,y)=[x+1,y+2]\mbox{ and }K_2(x,y)=[y+1,x+2].
\]
Consequently, the map $\mathcal{X}:[0,1]^2\tos\R^2$ is given by
\[
\mathcal{X}(x,y)=[x+1,y+2]\times[y+1,x+2].
\]
It is clear that $\mathcal{X}$ is not a self-map. Moreover, it does not have fixed point and consequently the GNEP associated to any two functions does not have solutions.

On the other hand, consider the functions 
$\theta_1,\theta_2:\R^2\to\R$ defined as
\[
\theta_1(x,y)=x^3-y\mbox{ and }\theta_2(x,y)=x+y^3.
\]
We define the maps $M_1,M_2:[0,1]^2\tos\R^2$ as
\[
M_1(x,y)=\{z\in [x+1,y+2]:~ \theta_1(z,y)\leq \theta_2(w,y),\mbox{ for all }w\in [x+1,y+2]\}
\]
and
\[
M_2(x,y)=\{z\in [y+1,x+2]:~ \theta_2(x,z)\leq \theta_2(x,w),\mbox{ for all }w\in [y+1,x+2]\}.
\]
Clearly $M_1(x,y)=\{x+1\}$ and $M_2(x,y)=\{y+1\}$. Thus, by considering the Euclidean norm in $\R^2$, we can see that $(1,1)$ is the only projected solution for the GNEP. Theorem \ref{MS} guarantees the existence of such a projected solution, contrary to Theorem 3.1 in \cite{BC-2021}, because the constraint maps are also dependent on its own player's strategy. Moreover, we cannot apply Theorem 9 in \cite{CZ-2019}.
\end{example}

Using the same idea proposed by Castellani \emph{et al.} in \cite{CGL-2023}, we can relax the compactness of each strategy set $K_\nu$. However, we need to consider a particular norm.

\begin{theorem}\label{MS2}
Assume the Euclidean norm in $\R^n$, and moreover for each player $\nu\in N$:
\begin{enumerate}
\item $K_\nu$ is convex, closed and non-empty subset of $\R^{n_\nu}$,
\item $X_\nu$ is continuous with convex, compact and non-empty values, 
\item $X_\nu(K)$ is bounded,
\item $\theta_\nu$ is pseudo-continuous and
\item $\theta_\nu(\cdot,x^{-\nu})$ is quasi-convex, for all $x^{-\nu}$;
\end{enumerate}
then there exists a projected solution.
\end{theorem}
\begin{proof}
The projection map $P_K$ is single-valued and consequently  it is continuous, see \cite{Deutsch1973}. 
By considering the sets $D$ and $C$, and the map $M$ in the proof of Theorem \ref{MS}.
Now we define the map $S:D\tos D$ as
\[
S(y)=M(P_K(y),y)
\]
which is clearly upper semicontinuous, due to Theorem in \cite{aliprantis06}. Moreover, it has convex, compact and non-empty values. Consequently, by Kakutani's theorem there exists a fixed point of $S$. Thus, it is enough to show this fixed point produces a projected solution. Indeed, let $\hat{y}$ be a fixed point of $S$. Thus, $\hat{y}^\nu\in M_\nu(\hat{x},\hat{y})$ for all $\nu\in N$, where $\hat{x}=P_K(\hat{y})$. This is equivalent to $\hat{y}\in\nep(\{\theta_\nu,X_\nu(\hat{x})\}_{\nu\in N})$. Therefore, $\hat{x}$ is a projected solution.
\end{proof}
\begin{remark}
In the above result, we can consider any norm such that the projection map $P_K$ is single-valued and continuous. 
\end{remark}
\subsection{The jointly convex case}

An important instance of generalized Nash equilibrium problem was presented by Rosen in \cite{Rosen}. More specifically, given  a convex and non-empty subset $X$ of $\R^n$, the aim of player $\nu\in N$ is to find $x^\nu$, given the strategy of rival players $x^{-\nu}$, such that it solves the problem
\begin{align}\label{RGNEP}
\min_{x^\nu} \theta_\nu(x^\nu,x^{-\nu}),~\mbox{ subject to }~(x^\nu,x^{-\nu})\in X.
\end{align}
A vector $\hat{x}\in X$ is a \emph{generalized Nash equilibrium in the sense of Rosen} if, for each player $\nu\in N$, $\hat{x}^{\nu}$ is a solution of the problem \eqref{RGNEP} associated to $\hat{x}^{-\nu}$.

\,

After the seminal paper of Rosen \cite{Rosen}, the authors in \cite{Aussel-Dutta} extended his existence result to the case of semi strict quasi-convexity,  Bueno \emph{et al.} \cite{BCC} dealt with the quasi-convexity case, and recently Calder\'on and Cotrina in \cite{CCJC2023} consider the noncompact case.

\,

Now, for each player $\nu\in N$, we consider the set $K_{\nu}$ as the projection of $X$ onto $\R^{n_\nu}$. Additionaly, for each $x\in X$ we consider the set
\[
X_\nu(x):=\{y^\nu\in\R^{n_\nu}:~(y^\nu,x^{-\nu})\in X\}.
\]
Thus, each $X_\nu$ is defined on $X$ onto $K_\nu$. 
This allows us to define the map 
$\mathcal{X}:X\tos\R^n$ by
\[
\mathcal{X}(x)=\prod_{\nu\in I_p}X_\nu(x),
\]
which is not a self-map in general. 
 Thus, a natural question arises: is a classical solution any projected solution?
Or in other words, we want to know if there exist $\hat{x}\in X$ and $\hat{y}\in \mathcal{X}(\hat{x})\setminus \{\hat{x}\}$ such that
\begin{itemize}
\item $\hat{x}$ is the projection of $\hat{y}$ on $X$, and
\item $\hat{y}\in \nep(\{\theta_\nu,X_\nu(\hat{x})\}_{\nu\in N})$.
\end{itemize}

We are ready for our main result of this subsection, which gives a positive answer to our question.

\begin{proposition}
Let $\|\cdot\|$ be a norm in  $\R^n$  and $p$ be the number of players. Then any projected solution is a classical solution.
\end{proposition}
\begin{proof}
Let $\hat{x}\in X$ be a projected solution, that means there exists $\hat{y}\in \mathcal{X}(\hat{x})$ such that
\begin{itemize}
\item $\|\hat{y}-\hat{x}\|\leq \|\hat{y}-x\|$, for all $x\in X$ and
\item $\theta_\nu(\hat{y})\leq \theta_\nu(y^\nu,\hat{y}^{-\nu})$ for all $y^\nu\in X_\nu(\hat{x})$.
\end{itemize}
First, notice that 
\[
\hat{y}=\sum_{\nu=1}^p\left((\hat{y}^\nu,\hat{x}^{-\nu})-(0,\hat{x}^{-\nu})\right)
\]
Now, for any $t\in\R$ we have
\begin{align*}
t\hat{x}+(1-t)\hat{y}&=t\hat{x}+(1-t)\sum_{\nu=1}^p\left((\hat{y}^\nu,\hat{x}^{-\nu})-(0,\hat{x}^{-\nu})\right)\\
&=t\hat{x}+(1-t)\sum_{\nu=1}^p(\hat{y}^\nu,\hat{x}^{-\nu})-(1-t)(p-1)\hat{x}\\
&=(t-(1-t)(p-1))\hat{x}+(1-t)\sum_{\nu=1}^p(\hat{y}^\nu,\hat{x}^{-\nu}).
\end{align*}
Clearly $t-(1-t)(p-1)+\sum_{\nu=1}^p(1-t)=1$. Thus, for any $1>t>\frac{p-1}{p}>0$, we deduce that
$z_t=t\hat{x}+(1-t)\hat{y}\in X$, due to this point belongs to the convex hull of $\hat{x}, (\hat{y}^1,\hat{x}^{-1}),\dots,(\hat{y}^p,\hat{x}^{-p})\in X.$ Consequently
\[
\|\hat{y}-z_t\|\leq  t\|\hat{y}-\hat{x}\|
\]
which in turn implies $\|\hat{y}-\hat{x}\|=0$. Hence, $\hat{y}=\hat{x}$.
\end{proof}

\subsection{Equivalence between Nash equilibrium theorems}

It was showed in \cite{BC-2021} the existence of projected solutions for GNEPs which are not generalized Nash equilibria. However, we will show that  the problem of finding projected solutions for GNEPs can be associated to a particular GNEP by adding a new player.

Assume that $N=\{1,2,\cdots,p\}$ and consider $M=N\cup\{p+1\}$. Thus, for each $\nu\in M$ we consider the sets $\hat{K}_\nu$ defined by
\[
\hat{K}_\nu=\begin{cases}
\co(K_\nu\cup X_\nu(K)),&\mbox{if }\nu\in N; \\
K,& \mbox{if }\nu=p+1 
\end{cases}
\]
As usual we write ${\bf x}=({\bf x}^\nu,{\bf x}^{-\nu})\in \hat{K}=\prod_{\nu\in M}\hat{K}_\nu$ in order to emphasize the strategy of player $\nu$. Moreover, we write ${\bf x}_0$ instead ${\bf x}^{-(p+1)}$. It is important to notice that for each $\nu\in N$
\[
{\bf x}^\nu={\bf x}_0^{\nu}.
\]
Let us define the map $\hat{X}_\nu:\hat{K}\tos\hat{K}_\nu$ and the function $\hat{\theta}_\nu:\R^n\times\R^n\to\R$, respectively, as
\[
\hat{X}_\nu({\bf x})=\begin{cases}
X_\nu({\bf x}^{p+1}),&\mbox{if }\nu\in N\\
K,&\mbox{if }\nu=p+1
\end{cases}
\mbox{ and }
\hat{\theta}_\nu({\bf x})=\begin{cases}
\theta_\nu({\bf x}_0),&\mbox{if }\nu\in N\\
\|{\bf x}_0-{\bf x}^{p+1}\|,&\mbox{if }\nu=p+1.
\end{cases}
\]
We denote the set of projected solutions by $\psgnep(\{\theta_\nu,X_\nu\}_{\nu\in N})$, and we establish  the relationship between the sets $\gnep(\{\hat{\theta}_\nu,\hat{X}_\nu\}_{\nu\in M})$ and $\psgnep(\{\theta_\nu,X_\nu\}_{\nu\in N})$.
\begin{proposition}\label{reformulation}
The following implications hold:
\begin{enumerate}
\item If $\hat{\bf x}\in \gnep(\{\hat{\theta}_\nu,\hat{X}_\nu\}_{\nu\in M})$, then $\hat{\bf x}^{p+1}\in\psgnep(\{\theta_\nu,X_\nu\}_{\nu\in N})$. 
\item If $\hat{x}\in\psgnep(\{\theta_\nu,X_\nu\}_{\nu\in N})$, then there exists $\hat{y}\in\R^n$ such that the vector $\hat{\bf x}=(\hat{y}, \hat{x})\in \gnep(\{\hat{\theta}_\nu,\hat{X}_\nu\}_{\nu\in M})$.
\end{enumerate}
\end{proposition}
\begin{proof}
\begin{enumerate}
\item  If $\hat{\bf x}\in \gnep(\{\hat{\theta}_\nu,\hat{X}_\nu\}_{\nu\in M})$, then for each $\nu\in M$
\begin{align}\label{Eqv}
\hat{\bf x}^\nu\in \argmin_{\hat{X}_\nu(\hat{\bf x})} \hat{\theta}_\nu(\cdot,\hat{\bf x}^{-\nu})
\end{align}
The previous relation \eqref{Eqv} is equivalent to  
\[
\hat{\bf x}_0^\nu\in \argmin_{X_\nu(\hat{\bf x}^{p+1})}\theta_\nu(\cdot, \hat{\bf x}_0^{-\nu}),\mbox{ for all }\nu\in N;
\]
and for $\nu=p+1$
\[
\|\hat{\bf x}_0-\hat{\bf x}^{p+1}\|\leq \|\hat{\bf x}_0-{\bf x}^{p+1}\|,\mbox{ for all }{\bf x}^{p+1}\in K.
\]
Since $P_K(\hat{\bf x}_0)=\argmin_{K}\|\cdot-\hat{\bf x}_0\|$, this last inequality implies $\hat{\bf x}^{p+1}\in P_K(\hat{\bf x}_0)$. Therefore, $\hat{\bf x}^{p+1}\in\psgnep(\{\theta_\nu,X_\nu\}_{\nu\in N})$. 
\item For $\nu\in N$ is trivial, and for $\nu=p+1$, the result follows from the fact that 
$P_K(\hat{y})=\argmin_{K}\|\cdot-\hat{y}\|$.
\end{enumerate}

\end{proof}

Now, we are in position to state the following result, which states that Theorem \ref{MS} can be deduced from Corollary \ref{cor-PS}.

\begin{theorem}
Corollary \ref{cor-PS} implies Theorem \ref{MS}.
\end{theorem}
\begin{proof}
For each $\nu\in N$ we have that 
\begin{itemize}
\item the set $\hat{K}_\nu$ is compact, convex and non-empty, due to the set $K_\nu$ is compact and non-empty, and the map $X_\nu$ is upper semicontinuous;
\item the map $\hat{X}_\nu$ is continuous with compact, convex and non-empty values, because the map $X_\nu$ is so;
\item the function $\hat{\theta}_\nu$ is pseudo-continuous and quasiconvex in its own variable, because $\theta_\nu$ is so.
\end{itemize}
For $p+1$, we have that $\hat{K}_{p+1}=K$, which is convex, compact and non-empty, 
the map $\hat{X}_{p+1}$ is constant and consequently it is continuous with convex, compact and non-empty values. Furthermore, since the norm is continuous and convex we obtain that  the function $\hat{\theta}_{p+1}$ is continuous and convex. 
Thus, by Corollary \ref{cor-PS} there exists at least one element of $\gnep(\{\hat{\theta}_\nu,\hat{X}_\nu\}_{\nu\in M})$. Hence, the result follows from Proposition \ref{reformulation}, part {\it 1.}. 
\end{proof}

\section*{Conclusions}
In this manuscript, we improve some existence results on projected solutions for generalized Nash equilibrium problems. We establish that the concept of projected solution coincides with the classical notion of generalized Nash equilibrium for generalized Nash games proposed by Rosen. Finally, we reformulate the problem of finding projected solutions for GNEPs as another GNEP by adding an extra player.
\bibliographystyle{abbrv}


\begin{thebibliography}{10}

\bibitem{aliprantis06}
C.~D. Aliprantis and K.~C. Border.
\newblock {\em Infinite Dimensional Analysis: a Hitchhiker's Guide}.
\newblock Springer-Verlag Berlin Heidelberg, 2006.

\bibitem{Arrow-Debreu}
K.~J. Arrow and G.~Debreu.
\newblock Existence of an equilibrium for a competitive economy.
\newblock {\em Econometrica}, 22(3):265--290, 1954.

\bibitem{Aussel-Dutta}
D.~Aussel and J.~Dutta.
\newblock {Generalized Nash Equilibrium Problem, Variational Inequality and
  Quasiconvexity}.
\newblock {\em Oper. Res. Lett.}, 36(4):461--464, 2008.

\bibitem{AGM-16}
D.~Aussel, R.~Gupta, and A.~Mehra.
\newblock {Evolutionary Variational Inequality Formulation of the Generalized
  Nash Equilibrium Problem}.
\newblock {\em J. Optim. Theory App.}, 169(1):74--90, 2016.

\bibitem{ASV-2016}
D.~Aussel, A.~Sultana, and V.~Vetrivel.
\newblock {On the Existence of Projected Solutions of Quasi-Variational
  Inequalities and Generalized Nash Equilibrium Problems}.
\newblock {\em J. Optim. Theory Appl.}, 170(3):818--837, 2016.

\bibitem{AVS21}
D.~Aussel, K.~C. Van, and D.~Salas.
\newblock Existence results for generalized nash equilibrium problems under
  continuity-like properties of sublevel sets.
\newblock {\em SIAM Journal on Optimization}, 31(4):2784--2806, 2021.

\bibitem{BCC}
O.~Bueno, C.~Calderon, and J.~Cotrina.
\newblock {A note on coupled constraint Nash games}, 2022.
\newblock Preprint on arXiv: 2201.04262.

\bibitem{BC-2021}
O.~Bueno and J.~Cotrina.
\newblock {Existence of Projected Solutions for Generalized Nash Equilibrium
  Problems}.
\newblock {\em J. Optim. Theory Appl.}, 191(1):344--362, 2021.

\bibitem{CCJC2023}
C.~Calderón and J.~Cotrina.
\newblock The generalized {Nash} game proposed by {Rosen}, 2023.
\newblock Preprint on arXiv: 2307.03532.

\bibitem{CGL-2023}
M.~Castellani, M.~Giuli, and S.~Latini.
\newblock {Projected solutions for finite-dimensional quasi-equilibrium
  problems}.
\newblock {\em Computational Management and Science}, 20, 2023.

\bibitem{Cavazzuti}
E.~Cavazzuti, M.~Pappalardo, and M.~Passacantando.
\newblock Nash equilibria, variational inequalities, and dynamical systems.
\newblock {\em Journal of Optimization Theory and Applications},
  114(3):491--506, 2002.

\bibitem{CK-2015}
J.~Contreras, J.~B. Krawczyk, and J.~Zuccollo.
\newblock Economics of collective monitoring: a study of environmentally
  constrained electricity generators.
\newblock {\em Computational Management Science}, 13(3):349--369, 2016.

\bibitem{Co21}
J.~Cotrina.
\newblock {Remarks on pseudo-continuity}.
\newblock {\em Minimax Theory and its Applications}, 2022.

\bibitem{CHS-2020}
J.~Cotrina, A.~Hantoute, and A.~Svensson.
\newblock Existence of quasi-equilibria on unbounded constraint sets.
\newblock {\em Optimization}, 0(0):1--18, 2020.

\bibitem{CZ-2018}
J.~{Cotrina} and J.~{Z{\'u}{\~n}iga}.
\newblock {Time-dependent generalized Nash equilibrium problem}.
\newblock {\em J. Optim. Theory Appl.}, 179:1054--1064, 2018.

\bibitem{CZ-2019}
J.~Cotrina and J.~Z{\'u}{\~n}iga.
\newblock {Quasi-equilibrium problems with non-self constraint map}.
\newblock {\em J. Glob Optim}, 75:177--197, 2019.

\bibitem{Dasgupta}
P.~Dasgupta and E.~Maskin.
\newblock The existence of equilibrium in discontinuous economic games, part i
  (theory).
\newblock {\em Review of Economic Studies}, 53(1):1--26, 1986.
\newblock Reprinted in K. Binmore and P. Dasgupta (eds.), Economic
  Organizations as Games, Oxford: Basil Blackwell, 1986, pp. 48-82.

\bibitem{Debreu}
G.~Debreu.
\newblock A social equilibrium existence theorem.
\newblock {\em Proceedings of the National Academy of Sciences of the United
  States of America}, 38(10):886--893, 1952.

\bibitem{Deutsch1973}
F.~Deutsch, W.~Pollul, and I.~Singer.
\newblock On set-valued metric projections, {Hahn-Banach} extension maps, and
  spherical image maps.
\newblock {\em Duke Math. J.}, 40(2):355--370, June 1973.

\bibitem{FACCHINEI2007159}
F.~Facchinei, A.~Fischer, and V.~Piccialli.
\newblock On generalized nash games and variational inequalities.
\newblock {\em Oper. Res. Let}, 35(2):159 -- 164, 2007.

\bibitem{HARKER199181}
P.~T. Harker.
\newblock {Generalized Nash games and quasi-variational inequalities}.
\newblock {\em European Journal of Operational Research}, 54(1):81 -- 94, 1991.

\bibitem{K-2005}
J.~B. Krawczyk.
\newblock Coupled constraint nash equilibria in environmental games.
\newblock {\em Resource and Energy Economics}, 27(2):157 -- 181, 2005.

\bibitem{K-2016}
J.~B. Krawczyk and M.~Tidball.
\newblock Economic problems with constraints: How efficiency relates to
  equilibrium.
\newblock {\em International Game Theory Review}, 18(04):1650011, 2016.

\bibitem{Lassonde}
M.~Lassonde.
\newblock Fixed points for kakutani factorizable multifunctions.
\newblock {\em J. Math. Anal. Appl.}, 152(1):46--60, Oct. 1990.

\bibitem{Le_Cadre2019-xc}
H.~Le~Cadre, P.~Jacquot, C.~Wan, and C.~Alasseur.
\newblock Peer-to-peer electricity market analysis: From variational to
  generalized nash equilibrium.
\newblock {\em Eur. J. Oper. Res.}, Sept. 2019.

\bibitem{MORGAN2007}
J.~Morgan and V.~Scalzo.
\newblock {Pseudocontinuous functions and existence of Nash equilibria}.
\newblock {\em J. Math. Econ.}, 43(2):174--183, 2007.

\bibitem{Nash}
J.~Nash.
\newblock Non-cooperative games.
\newblock {\em Annals of Mathematics}, 54(2):286--295, 1951.

\bibitem{Rabia2016}
R.~Nessah and G.~Tian.
\newblock On the existence of nash equilibrium in discontinuous games.
\newblock {\em Economic Theory}, 61(3):515--540, 2016.

\bibitem{Parilina2022-oi}
E.~Parilina, P.~V. Reddy, and G.~Zaccour.
\newblock {\em Theory and applications of dynamic games}.
\newblock Theory and Decision Library. Series C: Game Theory, Mathematical
  Programming and Operations Research. Springer International Publishing, Cham,
  Switzerland, 1 edition, Nov. 2022.

\bibitem{Reny}
P.~Reny.
\newblock On the existence of pure and mixed strategy {N}ash equilibria in
  discontinuous games.
\newblock {\em Econometrica}, \textbf{67}(5):1029--1056, 1999.

\bibitem{Rosen}
J.~B. Rosen.
\newblock Existence and uniqueness of equilibrium points for concave n-person
  games.
\newblock {\em Econometrica}, 33(3):520--534, 1965.

\bibitem{Vardar2019}
B.~Vardar and G.~Zaccour.
\newblock Strategic bilateral exchange of a bad.
\newblock {\em Operations Research Letters}, 47(4):235 -- 240, 2019.

\end{thebibliography}

\end{document}